\documentclass[12pt,a4paper]{amsart}
\usepackage{graphicx}
\usepackage{latexsym}
\usepackage{amsxtra}
\usepackage{amsmath}
\usepackage{amsfonts}
\usepackage{amssymb}
\usepackage[figuresright]{rotating}

\newtheorem{theorem}{Theorem}[section]
\newtheorem{corollary}[theorem]{Corollary}
\newtheorem{lemma}[theorem]{Lemma}
\newtheorem{proposition}[theorem]{Proposition}

\title{Quadratic embedding constants of path graphs}

\author{Wojciech M{\l}otkowski}
\address{Instytut Matematyczny,
Uniwersytet Wroc{\l}awski,
Plac~Grunwaldzki~2/4,
50-384 Wroc{\l}aw, Poland}
\email{mlotkow@math.uni.wroc.pl}

\begin{document}

\begin{abstract}
We characterize positive definiteness for some family of matrices.
As an application, we derive the explicit value of the quadratic embedding constants of the path graphs.
\end{abstract}

\maketitle

\section{Introduction}\label{sec:intro}

Let $G=(V,E)$ be a graph, with $V$ as the set of vertices and $E$ as the set of edges, i.e. two-element subsets of $V$.
We assume that $G$ is \textit{connected}, which means that for every $x,y\in V$ there exists a \textit{walk} connecting $x$ and $y$:
a finite sequence $x_0,x_1,\ldots,x_n\in V$ such that $n\ge0$, $x_0=x, x_n=y$ and $\{x_{i-1},x_i\}\in E$ for $i=1,\ldots,n$.
The distance $d(x,y)$ is defined as the smallest possible length $n$ of such a walk.
For some connected graphs the \textit{distance matrix} $(d(x,y))_{x,y\in V}$ is conditionally negative definite,
equivalently, there exists a function $\phi$ which maps $V$ into a Hilbert space $\mathcal{H}$
and satisfies
\[d(x,y)=\|\phi(x)-\phi(y)\|^{2}\]
for all $x,y\in V$. This motivated the authors of \cite{obata2018} to introduce and study the \textit{quadratic embedding constant},
defined as
\begin{equation}\label{def:qec}
\mathrm{QEC}(G):=\sup\left\{\sum_{x,y\in V}d(x,y)f(x)f(y):f\in\mathcal{F}_{0,1}(V)\right\},
\end{equation}
where $\mathcal{F}_{0,1}(V)$ denotes the set of all finitely supported functions $f:V\to\mathbb{R}$
satisfying $\sum_{x\in V}f(x)=0$ and $\sum_{x\in V}f(x)^2=1$.
In particular, the distance matrix  $(d(x,y))_{x,y\in V}$ is conditionally negative definite if and only if $\mathrm{QEC}(G)\le0$.
Several examples and properties were furnished in \cite{obata2021,obata2020,mlotkowskiobata2020,obata2017,obata2018}.
Applying the min-max theorem and the Perron-Frobenius theorem, one can observe
that if $V$ is finite and $\lambda_1(G)\ge\lambda_2(G)\ge\ldots\ge\lambda_{|V|}(G)$
are the eigenvalues of the distance matrix of $G$, then
\[
\lambda_2(G)\le\mathrm{QEC}(G)<\lambda_1(G).
\]

In this paper we will study finite \textit{path graphs}, i.e. graphs of the form $P_n:=(V,E)$,
where
\[
V:=\{1,2,\ldots,n\},\qquad
E:=\big\{\{1,2\},\{2,3\},\ldots,\{n-1,n\}\big\}.
\]
The eigenvalues $\lambda_i(P_n)$ of the distance matrix of $P_n$ were found in \cite{ruzieh1990}.
In particular,
\[
\lambda_2(P_n)=\left\{\begin{array}{ll}
\displaystyle\frac{-1}{1+\cos(\pi/n)}\phantom{\frac{a}{\frac{a}{\frac{b}{c}}}}&\hbox{if $n$ is even,}\\
\displaystyle\frac{-1}{1-\cos\theta^*}&\hbox{if $n$ is odd,}
\end{array}
\right.
\]
where $\theta^*$ is the maximal solution of the equation:
\[
\tan(\theta/2)\tan(n\theta/2)=-1/n,\qquad
\theta\in(0,\pi).
\]

It was observed in \cite[Proposition~5.4]{mlotkowskiobata2020} that $\mathrm{QEC}(P_n)$ is equal to the minimal $t$
such that the matrix
\begin{equation}\label{matrixtt}
\left[2\min\{i,j\}+t+t\cdot\delta_{i,j}\right]_{i,j=1}^{n-1}
\end{equation}
is positive definite, i.e. all the eigenvalues of this matrix are nonnegative.
This led to inequality: $\mathrm{QEC}(P_n)\le-1/2$, see \cite[Theorem~5.6]{mlotkowskiobata2020}.
Our aim is to provide the exact value:

\begin{theorem}\label{th:pathgraphs}
For $n\ge2$ we have
\begin{equation}
\mathrm{QEC}(P_n)=\frac{-1}{1+\cos(\pi/n)}.
\end{equation}
Consequently, if $n$ is even then $\lambda_2(P_n)=\mathrm{QEC}(P_n)$
and if $n$ is odd then $\lambda_2(P_n)<\mathrm{QEC}(P_n)$.
\end{theorem}

As an immediate consequence we obtain the quadratic embedding constant for $\mathbb{N}$ and $\mathbb{Z}$
regarded as infinite path graphs, with edges $\{i,i+1\}$, $i\in\mathbb{N}$ or $i\in\mathbb{Z}$,
c.f. \cite[Theorem~5.7]{mlotkowskiobata2020}.

\begin{corollary}
$\mathrm{QEC}(\mathbb{N})=\mathrm{QEC}(\mathbb{Z})=-1/2$.
\end{corollary}

The paper is organized as follows.
First we examine a family of auxiliary polynomials $S_n(a,b;t)$, $a,b\in\mathbb{R}$, $n\ge0$.
Section~\ref{sec:matrices} is devoted to the study of a family $A_{n}(s,t)$ of matrices,
a two-parameter version of~(\ref{matrixtt}).
We provide formula for $\det A_{n}(s,t)$, $n<\infty$, in terms of polynomials $S_n(a,b;t)$,
and characterize these $s,t\in\mathbb{R}$ for which $A_{n}(s,t)$, $1\le n\le\infty$, is positive definite.
Finally we prove Theorem~\ref{th:pathgraphs}.

\section{A family of polynomials}

Now we are going to study a family of polynomials defined
by the following recurrence: $S_0(a,b;t):=1$, $S_1(a,b;t):=at+b$ and
\begin{equation}\label{eq:defsn}
S_{n}(a,b;t)=(1+2t)S_{n-1}(a,b;t)-t^2 S_{n-2}(a,b;t)
\end{equation}
for $n\ge2$, $a,b,t\in\mathbb{R}$.

\begin{proposition}\label{prop:snformula}
For $n\ge0$ we have $\deg S_n(a,b;t)\le n$ and the coefficient at $t^n$
in $S_{n}(a,b;t)$ is equal to $an-n+1$. Moreover,
\begin{multline}\label{for:snabt}
S_n(a,b;t)=(at+b)\sum_{k=0}^{n-1}\binom{2n-k-1}{k}t^k-\sum_{k=0}^{n-2}\binom{2n-k-3}{k}t^{k+2}\\
=\frac{1}{2^{n+1}\sqrt{1+4t}}\left[\left(2b-1+2(a-1)t+\sqrt{1+4t}\right)\left(1+2t+\sqrt{1+4t}\right)^n\right.\\
\left.-\left(2b-1+2(a-1)t-\sqrt{1+4t}\right)\left(1+2t-\sqrt{1+4t}\right)^n\right].
\end{multline}
The former formula holds for $n\ge1$, while the latter for $n\ge0$ and $t\ne-1/4$.
For $t=-1/4$, $n\ge0$ we have
\begin{equation}\label{for:minus14}
S_{n}(a,b;-1/4)=\frac{1}{4^n}\left(4nb-na-n+1\right).
\end{equation}
\end{proposition}

\begin{proof}
The first statement can be easily proved by induction.
Now define:
\begin{equation}\label{eq:wn}
W_{n}(t):=\sum_{k=0}^{n}\binom{2n-k+1}{k}t^k
\end{equation}
(see $A172431$ in OEIS~\cite{oeis}).
Then $W_{0}(t)=1$, $W_{1}(t)=2t+1$, and one can check that
\[
W_{n}(t)=(1+2t)W_{n-1}(t)-t^2 W_{n-2}(t)
\]
for $n\ge2$, so that $W_{n}(t)=S_{n}(2,1;t)$. Putting $Q_{0}(t):=1$ and
\[
Q_{n}(t):=(at+b)W_{n-1}(t)-t^2 W_{n-2}(t)
\]
for $n\ge1$ ($W_{-1}(t):=0$) we have $Q_{1}(t)=at+b$ and $Q_{n}(t)=(1+2t)Q_{n-1}(t)-t^2 Q_{n-2}(t)$
for $n\ge1$, so $Q_{n}(t)$ coincides with $S_{n}(a,b;t)$ and the first formula in (\ref{for:snabt}) holds.
Moreover, it is easy to verify by induction that $W_n(-1/4)=(n+1)/4^n$, which leads to~(\ref{for:minus14}).

For $n\ge0$, $a,b\in\mathbb{R}$, $t\ne-1/4$ put
\begin{multline*}
\phantom{u}
T_n(t):=\frac{1}{2^{n+1}\sqrt{1+4t}}\left[\left(2b-1+2(a-1)t+\sqrt{1+4t}\right)\left(1+2t+\sqrt{1+4t}\right)^n\right.\\
\left.-\left(2b-1+2(a-1)t-\sqrt{1+4t}\right)\left(1+2t-\sqrt{1+4t}\right)^n\right].
\phantom{u}
\end{multline*}
Then $T_{0}(t)=1$, $T_1(t)=at+b$ and from the identity
\[
\left(1+2t\pm\sqrt{1+4t}\right)^2=2(1+2t)\left(1+2t\pm\sqrt{1+4t}\right)-4t^2
\]
we have
\[
T_{n}(t)=(1+2t)T_{n-1}(t)-t^2 T_{n-2}(t)
\]
for $n\ge2$, which implies that $T_{n}(t)=S_{n}(a,b;t)$ for $n\ge0$, $t\ne-1/4$.
\end{proof}

For later use, we record the following identities:

\begin{proposition}
For $W_n(t)$ defined in (\ref{eq:wn}) we have:
\begin{equation}\label{eq:wnsn}
W_n(t)=S_n(2,1;t)=S_n(1,t+1;t).
\end{equation}
Moreover,
\begin{equation}\label{eq:s21}
W_{n}(t)=\frac{1}{2^{n+1}\sqrt{1+4t}}\left[\left(1+2t+\sqrt{1+4t}\right)^{n+1}
-\left(1+2t-\sqrt{1+4t}\right)^{n+1}\right].
\end{equation}
\end{proposition}

\begin{proof}
The former equality in (\ref{eq:wnsn}) was noted in the previous proof,
the latter, as well as (\ref{eq:s21}), is a consequence of~(\ref{for:snabt}).
\end{proof}

In some particular cases we are able to find the roots of $S_{n}(a,b;t)$.

\begin{proposition}\label{propositiontrigonometric}
For $n\ge1$ we have
\begin{align}
\frac{1}{n+1}S_{n}(2,1;t)&=\prod_{k=1}^{n}\left(t+\frac{1}{2+2\cos\left(\frac{k\pi}{n+1}\right)}\right),\label{trigall}\\
S_{n}(1,1/2;t)&=\prod_{k=1}^{n}\left(t+\frac{1}{2+2\cos\big(\frac{(2k-1)\pi}{2n}\big)}\right),\\
S_{n}(1,1;t)&=\prod_{k=1}^{n}\left(t+\frac{1}{2+2\cos\left(\frac{2k\pi}{2n+1}\right)}\right),\\
\frac{1}{2n+1}S_{n}(3,1;t)&=\prod_{k=1}^{n}\left(t+\frac{1}{2+2\cos\big(\frac{(2k-1)\pi}{2n+1}\big)}\right).
\end{align}
\end{proposition}

\begin{proof}
First we note that in view of the first part of Proposition~\ref{prop:snformula},
if $a\ge1$ then $\deg S_{n}(a,b;t)=n$ and the leading term is~$(an-n+1)t^n$.

Fix $a,b,t\in\mathbb{R}$, with $1+4t<0$, and put
\begin{align*}
w&:=2b-1+2(a-1)t+\sqrt{1+4t},\\
z&:=1+2t+\sqrt{1+4t},
\end{align*}
with $\mathrm{Im}(w)>0$,  $\mathrm{Im}(z)>0$, $\alpha:=\mathrm{arg}(w)$, $\beta:=\mathrm{arg}(z)$, $0<\alpha,\beta<\pi$.
Then
\[
\cos\beta=\frac{1+2t}{-2t}\qquad\hbox{and hence}\qquad t=\frac{-1}{2+2\cos\beta}.
\]
In view of~(\ref{for:snabt}), we have $S_n(a,b;t)=0$ if and only if $w\cdot z^n=\overline{w}\cdot \overline{z}^n$,
equivalently, $\alpha+n\beta=k\pi$ for some $k=1,2,\ldots,n$.
Now we consider four cases.

1. If $a=2$, $b=1$ then $w=z$, $\alpha=\beta$ and therefore $\beta=k\pi/(n+1)$.

2. If $a=1$, $b=1/2$ then $\alpha=\pi/2$ and then $\beta=(2k-1)\pi/(2n)$.

3. If $a=b=1$ then
\[
\cos\alpha=\frac{1}{\sqrt{-4t}}=\sqrt{\frac{1+\cos\beta}{2}}=\cos(\beta/2),
\]
which implies $\alpha=\beta/2$ and $\beta=2k\pi/(2n+1)$.

4. Finally, if $a=3$, $b=1$ then
\begin{multline*}
\phantom{aaa}
\cos\alpha=\frac{1+4t}{\sqrt{4t(1+4t)}}=-\sqrt{\frac{1+4t}{4t}}=-\sqrt{\frac{1-\cos\beta}{2}}\\
=-\sin(\beta/2)=\cos(\pi/2+\beta/2),
\phantom{aaa}
\end{multline*}
therefore $\alpha=\pi/2+\beta/2$ and, consequently, $\beta=(2k-1)\pi/(2n+1)$.
\end{proof}

\begin{lemma}\label{lem:t2}
For $n\ge1$, $a,b\in\mathbb{R}$ we have
\begin{equation}
S_{n}(2,1;t)\cdot S_{n}(a,b;t)-S_{n-1}(2,1;t)\cdot S_{n+1}(a,b;t)=t^{2n}.
\end{equation}
\end{lemma}

\begin{proof}
Putting
\[
u_{\pm}:=2b-1+2(a-1)t\pm\sqrt{1+4t},\qquad
v_{\pm}:=1+2t\pm\sqrt{1+4t}
\]
we have
\begin{multline*}
\phantom{uu}
\left(v_{+}^{n+1}-v_{-}^{n+1}\right)\left(u_{+}v_{+}^{n}-u_{-}v_{-}^{n}\right)
-\left(v_{+}^{n}-v_{-}^{n}\right)\left(u_{+}v_{+}^{n+1}-u_{-}v_{-}^{n+1}\right)\\
=(u_{+}-u_{-})(v_{+}-v_{-})(v_{+} v_{-})^{n}=4(1+4t)\left(4t^2\right)^n,
\phantom{uu}
\end{multline*}
and, by (\ref{for:snabt}), the formula follows.
\end{proof}

Define
\begin{equation}\label{eq:tn}
t_n:=\left\{\begin{array}{ll}
-\infty&\hbox{if $n=1$,}\\
\displaystyle\frac{-1}{2+2\cos(\pi/n)}&\hbox{if $n\ge2$.}
\end{array}
\right.
\end{equation}

\begin{lemma}\label{lemmasigns0}
If $t_n<t<t_{n+1}$ then $S_{n-1}(2,1;t)>0$, $S_{n}(2,1;t)<0$
and if $t>t_{n+1}$ then $S_{n-1}(2,1;t)>0$, $S_{n}(2,1;t)>0$.
\end{lemma}

\begin{proof}
First note that the function $\alpha\mapsto-1/(2+2\cos\alpha)$ is decreasing with $\alpha\in(0,\pi)$.
Therefore the statement is a consequence of (\ref{trigall}), because for $n\ge2$ we have
\[
\frac{-1}{2+2\cos\big(2\pi/(n+1)\big)}<\frac{-1}{2+2\cos\big(\pi/n\big)}<\frac{-1}{2+2\cos\big(\pi/(n+1)\big)}.
\]
\end{proof}

Now we collect properties of polynomials of the form~$S_{n}(1,s+1;t)$, which will be applied in the next section.

\begin{proposition}
For $s\in\mathbb{R}$, $n\ge1$, we have
\begin{multline}\label{le:s1esplus1}
\phantom{uu}
S_{n}(1,s+1;t)=S_{n}(1,1;t)+s\cdot S_{n-1}(2,1;t)\\
=\sum_{k=0}^{n}\binom{2n-k}{k}t^k+s\sum_{k=0}^{n-1}\binom{2n-1-k}{k}t^k\\
=\frac{1}{2^{n+1}\sqrt{1+4t}}
\left[\left(1+2s+\sqrt{1+4t}\right)\left(1+2t+\sqrt{1+4t}\right)^n\right.\\
\left.-\left(1+2s-\sqrt{1+4t}\right)\left(1+2t-\sqrt{1+4t}\right)^n\right].
\phantom{uu}
\end{multline}
The latter formula is valid for $t\ne-1/4$, while
\begin{equation}\label{formuladetminus14}
S_{n}(1,s+1;-1/4)=\frac{4ns+2n+1}{4^n}.
\end{equation}
\end{proposition}

\begin{proof}
The first equality in (\ref{le:s1esplus1}) can be verified by induction: putting
$P_n(t):=S_{n}(1,1;t)+s\cdot S_{n-1}(2,1;t)$ we have $P_0(t)=1$, $P_1(t)=t+1+s$
and $P_{n}(t)=(1+2t)P_{n-1}(t)-t^2 P_{n-2}(t)$ for $n\ge2$, consequently
$P_{n}(t)=S_{n}(1,s+1;t)$ for all $n\ge0$. Now (\ref{le:s1esplus1}) and (\ref{formuladetminus14})
are consequences of (\ref{for:snabt}) and (\ref{for:minus14}).
\end{proof}

\section{A family of matrices}\label{sec:matrices}

Define a family of matrices
\begin{equation}
A_n(s,t):=\left[\min\{i,j\}+s+t\delta_{i,j}\right]_{i,j=1}^{n},
\end{equation}
where $s,t\in\mathbb{R}$, $1\le n\le\infty$.
These matrices seem interesting on their own, for the sake of Theorem~\ref{th:pathgraphs}
we are particularly interested in the case~$s=t$.
We are going to study their determinants and positive definiteness.
Note that if $s_1\le s_2$, $t_1\le t_2$, $n_1\ge n_2$ and $A_{n_1}(s_1,t_1)$ is positive definite then so is $A_{n_2}(s_2,t_2)$.

\begin{theorem}\label{theoremfinite}
For $s,t\in\mathbb{R}$, $n\in\mathbb{N}$, we have
\begin{equation}\label{thfiniteformula}
\det A_{n}(s,t)=S_{n}(1,s+1;t)=S_{n}(1,1;t)+s\cdot S_{n-1}(2,1;t).
\end{equation}

With the notation of (\ref{eq:tn}), the matrix $A_n(s,t)$ is positive definite if and only if
\begin{equation}\label{eq:posdef}
t> t_n \qquad\hbox{and}\qquad
S_{n}(1,s+1;t)\ge0.
\end{equation}
\end{theorem}

\begin{proof}
It is easy to verify (\ref{thfiniteformula}) for $n=1,2$.
Now assume that $n\ge3$ and let $\mathbf{k}_j$ denote the $j$th column of $A_{n}(s,t)$. Then
\[
\det A_n(s,t)=\det(\mathbf{k}_1,\ldots,\mathbf{k}_n)
=\det(\mathbf{k}_1,\ldots,\mathbf{k}_{n-1},\mathbf{k}_{n}-\mathbf{k}_{n-1})
\]
and, denoting the transposition by ``$\mathrm{T}$'', we have
\[
\mathbf{k}_{n}-\mathbf{k}_{n-1}=(0,\ldots,0,-t,1+t)^{\mathrm{T}}.
\]
Expanding the determinant along the last column $\mathbf{k}_{n}-\mathbf{k}_{n-1}$,
we get
\[
\det A_n(s,t)=(1+t)\det A_{n-1}(s,t)+t\det B,
\]
where
\[
B=\left(\mathbf{k}'_1,\ldots,\mathbf{k}'_{n-2},\mathbf{k}'_{n-1}+\eta\right),
\]
$\mathbf{k}'_{j}$ is the $j$th column of $A_{n-1}(s,t)$ and $\eta:=(0,\ldots,0,-t)^{\mathrm{T}}$.
This yields
\[
\det B=\det A_{n-1}(s,t)-t\det A_{n-2}(s,t),
\]
and, consequently,
\[
\det A_n(s,t)=(1+2t)\det A_{n-1}(s,t)-t^2 \det A_{n-2}(s,t),
\]
which completes the proof of~(\ref{thfiniteformula}).

Fix $t'\in\mathbb{R}$ and put
\[
\ell_k(s):=\det A_{k}(s,t'):=b_k+s\cdot a_k.
\]
Then, by Lemma~\ref{lem:t2}, we have
\begin{equation}\label{eq:akbk}
a_k b_{k-1}-a_{k-1} b_k=(t')^{2k-2}.
\end{equation}

With the notation of (\ref{eq:tn}), if $t_{k-1}<t'\le t_k$ then $a_{k-1}>0$, $a_{k}\le0$.
Consequently, the lines $\ell_{k-1},\ell_{k}$ cross at some point $(s',y')$, i.e.
$\ell_{k-1}(s')=\ell_{k}(s')=y'$, and, by (\ref{eq:akbk}),
\[
y'=\frac{a_{k} b_{k-1}-a_{k-1} b_{k}}{a_{k}-a_{k-1}}<0.
\]
This implies, that for every $s\in\mathbb{R}$ we have either $\ell_{k-1}(s)<0$ or $\ell_{k}(s)<0$,
therefore the matrix $A_{n}(s,t')$ is not positive definite for $n\ge k$.

Now, if we assume that $t'>t_n$ then $a_1,\ldots,a_n>0$.
Defining $s_k$ by $\ell_k(s_k)=0$, we have $s_1\le s_2\le\ldots\le s_n$, by~(\ref{eq:akbk}).
Therefore, if $\det A_{n}(s,t')>0$ then $s>s_n$, $\det A_{k}(s,t')>0$ for all $k\le n$,
and hence $A_n(s,t')$ is positive definite. If $\det A_{n}(s,t')=0$
then $A_n(s,t')$ is positive definite as pointwise limit of positive definite matrices.
The proof is complete.
\end{proof}

Note, that in particular,
\begin{equation}
\det A_n(t,t)=S_n(1,t+1;t)=S_{n}(2,1)=W_{n}(t).
\end{equation}

Since $S_n(1,s_0+1;t)$ is the characteristic polynomial of a real symmetric matrix, we have

\begin{corollary}
For every fixed $s_0\in\mathbb{R}$, $n\ge1$, the polynomial $S_{n}(1,s_0+1;t)$ has only real roots.
\end{corollary}

Now we consider some particular cases.

\begin{theorem}\label{theoremroots}
\phantom{We have the following:}
\begin{itemize}
\item The matrix $A_n(t,t)$ is positive definite if and only if
\[
t\ge\frac{-1}{2+2\cos\big(\pi/(n+1)\big)}.
\]
\item The matrix $A_n(-1/2,t)$ is positive definite if and only if
\[
t\ge\frac{-1}{2+2\cos\big(\pi/(2n)\big)}.
\]
\item The matrix $A_n(0,t)$ is positive definite if and only if
\[
t\ge\frac{-1}{2+2\cos\big(2\pi/(2n+1)\big)}.
\]
\item The matrix $A_n(2t,t)$ is positive definite if and only if
\[
t\ge\frac{-1}{2+2\cos\big(\pi/(2n+1)\big)}.
\]
\end{itemize}
\end{theorem}

\begin{proof}
First note that by (\ref{for:snabt}) we have $S_n(1,2t+1;t)=S_n(3,1;t)$.
It is easy to check the statements for $n=1$.
For $0<k<m$ put
\[
t^{k}_{m}:=\frac{-1}{2+2\cos(k\pi/m)}.
\]
Then, for $n>1$, we have
$t^{2}_{n+1}<t^{1}_{n}<t^{1}_{n+1}$,
$t^{3}_{2n}< t^{1}_{n}<t^{1}_{2n}$,
$t^{4}_{2n+1}< t^{1}_{n}<t^{2}_{2n+1}$,
$t^{3}_{2n+1}< t^{1}_{n}<t^{1}_{2n+1}$, and
the statements follow from Theorem~\ref{theoremfinite} and
Proposition~\ref{propositiontrigonometric}.
\end{proof}

Now we are able to describe the case $n=\infty$.

\begin{theorem}\label{theoreminfinite}
The infinite matrix $A_{\infty}(s,t)$ is positive definite if and only if
\[
1+4t\ge0\qquad\hbox{and}\qquad 1+2s+\sqrt{1+4t}\ge0.
\]
\end{theorem}

\begin{proof}
If $1+4t\ge0$ and $1+2s+\sqrt{1+4t}\ge0$ then all the matrices $A_n(s,t)$, $1\le n<\infty$, are positive definite
by (\ref{thfiniteformula}), (\ref{le:s1esplus1})
and (\ref{formuladetminus14}), hence so is~$A_{\infty}(s,t)$.

Now assume that $A_{\infty}(s,t)$ is positive definite.
Then $t>-1/\big(2+2\cos(\pi/n)\big)$ for every $n\in\mathbb{N}$, which implies $t\ge-1/4$.
If $t=-1/4$ then $s\ge-1/2$ in view of (\ref{formuladetminus14}).
If $1+4t>0$ then, putting
\[
q:=\frac{1+2t-\sqrt{1+4t}}{1+2t+\sqrt{1+4t}},
\]
we have $0<q<1$ and, in view of the last formula in (\ref{le:s1esplus1}), the inequality
\[
1+2s+\sqrt{1+4t}\ge\left(1-2s+\sqrt{1+4t}\right)q^n
\]
holds for every $n\ge1$. This implies that $1+2s+\sqrt{1+4t}\ge0$.
\end{proof}

\begin{figure}[htbp]
 \centering
  \includegraphics[width=100mm]{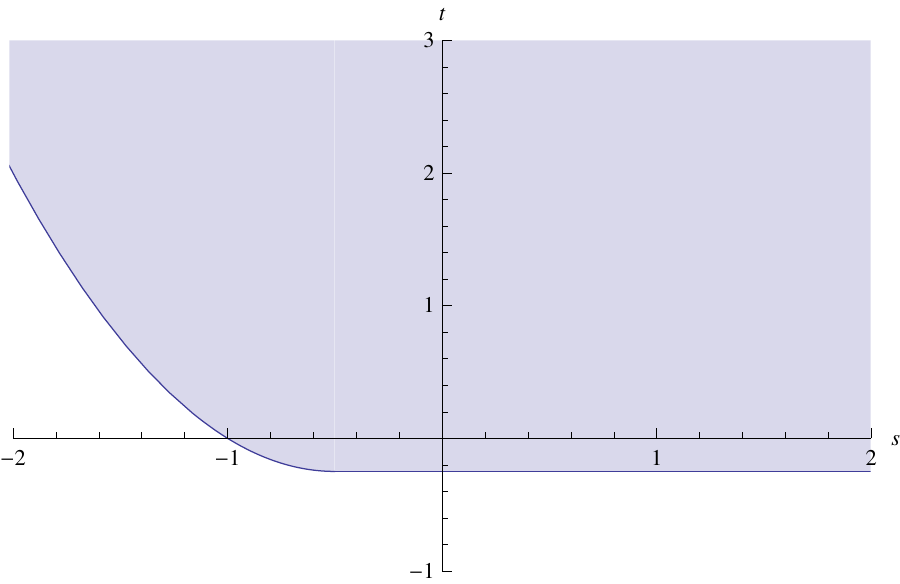}
 \caption{The range of positive definiteness of $A_{\infty}(s,t)$}
 \label{fig:one}
\end{figure}

For example, the following matrix is positive definite:
\[
4A_{\infty}(-1/2,-1/4)=
\left(
\begin{array}{cccccc}
1&2&2&2&2&\ldots\\
2&5&6&6&6&\ldots\\
2&6&9&10&10&\ldots\\
2&6&10&13&14&\ldots\\
2&6&10&14&17&\ldots\\
\vdots&\vdots&\vdots&\vdots&\vdots&\ddots
\end{array}
\right).
\]

\section{Path graphs}\label{sec:proofs}

Now we are ready to prove the main result of this paper.

\begin{proof}[Proof of Theorem~\ref{th:pathgraphs}]
In view of \cite[Proposition~5.4]{mlotkowskiobata2020}, $\mathrm{QEC}(P_n)$
is equal to the minimal~$t$ such that $A_{n-1}(t/2,t/2)$ is positive definite.
It remains to apply the first part of Theorem~\ref{theoremroots}.
\end{proof}

Now we indicate explicitly a vector $\mathbf{x}_n$ on the vertices of $P_n$ for which the supremum
in (\ref{def:qec}) is attained.
Put
\begin{equation}\label{eq:eigenvector}
\mathbf{x}_{n}:=\left(x_{n,1},x_{n,2},\ldots,x_{n,n}\right),\quad\hbox{with}\quad
x_{n,i}:=(-1)^{i}\sin\left(\frac{2i-1}{2n}\pi\right).
\end{equation}
If $n$ is even then $\mathbf{x}_{n}$ is the eigenvector corresponding to the second eigenvalue $\lambda_2(P_n)$,
see~\cite{ruzieh1990}.

\begin{proposition}
For $n\ge2$ we have
\begin{equation}\label{sumx}
\sum_{i=1}^{n} x_{n,i}=0,\qquad
\sum_{i=1}^{n}x_{n,i}^2=\frac{n}{2}
\end{equation}
and
\begin{equation}\label{sumxx}
\sum_{i,j=1}^{n}|i-j| x_{n,i}\cdot x_{n,j}=\frac{-1}{1+\cos(\pi/n)}\cdot\frac{n}{2}.
\end{equation}
\end{proposition}

\begin{proof}
Denoting by ``$\mathbf{i}$" the imaginary unit, we have
\[
\sum_{i=1}^{n}(-1)^i \exp\left(\frac{2i-1}{2n}\pi\mathbf{i}\right)
=-\exp\left(\pi\mathbf{i}/(2n)\right)\frac{1+(-1)^{n}}{1+\exp\left(\pi\mathbf{i}/n\right)}
=\frac{-1-(-1)^{n}}{2\cos\big(\pi/(2n)\big)}.
\]
Taking the imaginary part we get the first equation in (\ref{sumx}).
Similarly, since
\[
\sum_{i=1}^{n}\exp\left(\frac{(2i-1)\pi}{n}\mathbf{i}\right)=0,
\]
we get
\[
2\sum_{i=1}^{n}\sin^2\left(\frac{(2i-1)\pi}{2n}\right)
=\sum_{i=1}^{n}\left[1-\cos\left(\frac{(2i-1)\pi}{n}\right)\right]=n,
\]
which completes the proof of (\ref{sumx}).

Now we will prove (\ref{sumxx}). For $1\le k< n$ we have
\begin{multline*}
\phantom{uuu}
2\sum_{i=1}^{n-k}\sin\left(\frac{2i-1}{2n}\pi\right)\sin\left(\frac{2i+2k-1}{2n}\pi\right)\\
=\sum_{i=1}^{n-k}\left[\cos(k\pi/n)-\cos\big((2i+k-1)\pi/n\big)\right]\\
=(n-k)\cos(k\pi/n)+\frac{\sin(k\pi/n)}{\sin(\pi/n)}.
\phantom{uuu}
\end{multline*}
Therefore
\begin{multline}\label{eq:suma}
\phantom{uuu}
\sum_{i,j=1}^{n}|i-j|x_{n,i}\cdot x_{n,j}
=2\sum_{k=1}^{n-1}k\sum_{i=1}^{n-k}x_{n,i}\cdot x_{n,i+k}\\
=\sum_{k=1}^{n-1}(-1)^k k(n-k)\cos(k\pi/n)+\sum_{k=1}^{n-1}(-1)^{k}\frac{k\sin(k\pi/n)}{\sin(\pi/n)}.
\phantom{uuu}
\end{multline}
Now applying elementary formulas:
\begin{align*}
\sum_{k=1}^{n-1} kq^k&=\frac{q(1-q^n)-nq^n(1-q)}{(1-q)^2},\\
\sum_{k=1}^{n-1} k(n-k)q^k&=\frac{nq(1-q)(1+q^n)-q(1+q)(1-q^n)}{(1-q)^3}
\end{align*}
to $q:=-\exp(\pi\mathbf{i}/n)$, so that $q^n=-(-1)^n$, we obtain
\[
\sum_{k=1}^{n-1} k(-1)^k\exp(k\pi\mathbf{i}/n)
=\frac{-1-(-1)^n+(-1)^{n} n\big(\exp(-\pi\mathbf{i}/n)+1\big)}{4\cos^2\big(\pi/(2n)\big)},
\]
\begin{multline*}
\phantom{uuuu}
\sum_{k=1}^{n-1} k(n-k)(-1)^k\exp(k\pi\mathbf{i}/n)\\
=\frac{2n\cos\big(\pi/(2n)\big)\big((-1)^n-1\big)+2\sin\big(\pi/(2n)\big)\big(1+(-1)^n\big)\mathbf{i}}
{8\cos^3\big(\pi/(2n)\big)}.
\phantom{uuuu}
\end{multline*}
Consequently,
\begin{align*}
\sum_{k=0}^{n}\frac{(-1)^k k\sin(k\pi/n)}{\sin(\pi/n)}&=\frac{-(-1)^n n}{4\cos^2(\pi/(2n))},\\
\sum_{k=0}^{n}(-1)^k k(n-k)\cos(k\pi/n)&=\frac{\left((-1)^n-1\right)n}{4\cos^2(\pi/(2n))},
\end{align*}
which, together with (\ref{eq:suma}), leads to (\ref{sumxx}).
\end{proof}

\section*{Acknowledgements}

The author is grateful to Nobuaki Obata for useful discussions, in particular for pointing out reference~\cite{ruzieh1990},
and also to the anonymous referee for careful reading the paper and useful comments,
in particular for simplifying the proof of Theorem~\ref{theoremfinite}.


\begin{thebibliography}{99}

\bibitem{obata2021}
E. T. Baskoro, N. Obata,
\textit{Determining finite connected graphs along the quadratic embedding constants of paths,}
Electron. J. Graph Theory Appl. (EJGTA) \textbf{9} (2021), no. 2, 539--560.

\bibitem{obata2020}
Zhenzhen Lou, Nobuaki Obata, Qiongxiang Huang,
\textit{Quadratic Embedding Constants of Graph Joins,} arXiv 2020.

\bibitem{mlotkowskiobata2020}
W. M{\l}otkowski, N. Obata, \textit{On quadratic embedding constants of star product graphs,}
Hokkaido Math. J. \textbf{49} (2020), no. 1, 129--163.

\bibitem{obata2017}
N. Obata,
\textit{Quadratic embedding constants of wheel graphs,}
Interdiscip. Inform. Sci. \textbf{23} (2017), no. 2, 171--174.

\bibitem{obata2018}
N. Obata, A. Y. Zakiyyah,
\textit{Distance matrices and quadratic embedding of graphs,}
Electron. J. Graph Theory Appl. (EJGTA) \textbf{6} (2018), no. 1, 37--60.

\bibitem{oeis}
\textit{The On-Line Encyclopedia of Integer Sequences} (OEIS),   http://oeis.org/

\bibitem{ruzieh1990}
S.~N.~Ruzieh, D.~L.~Powers,
\textit{The distance spectrum of the path $P_n$ and the first distance eigenvector of connected graphs,}
Linear and Multilinear Algebra, \textbf{28} (1990), 75--81.


\end{thebibliography}
\end{document}